\documentclass[conference]{IEEEtran}

\ifCLASSINFOpdf

\else

\fi

\usepackage{amsfonts}
\usepackage{graphicx,subfig}
\usepackage{amsmath,amsthm}
\usepackage{wrapfig}
\usepackage{verbatim}
\usepackage{url}

\usepackage{epsfig}
\usepackage{amsmath}
\usepackage{amssymb}
\usepackage{psfrag}
\usepackage[absolute]{textpos}

\newtheorem{thm}{Theorem}
\newtheorem{lem}{Lemma}

\def\x{{\bf x}}
\def\y{{\bf y}}
\def\w{{\bf w}}

\newcommand{\ri}{\right}
\newcommand{\li}{\left}

\newcommand{\A}{\mathcal{A}}

\newcommand{\N}{\mathcal{N}}
\newcommand{\F}{\mathcal{F}}
\newcommand{\PP}{{\bf{P}}}
\newcommand{\QQ}{{\bf{Q}}}
\newcommand{\SSS}{{\bf{S}}}

\newcommand{\R}{\mathbb{R}}

\newcommand{\s}{\star}
\newcommand{\eps}{\epsilon}
\newcommand{\sig}{\sigma}

\newcommand{\z}{{\bf z}}

\newcommand{\td}{\text{diag}}

\newcommand{\sort}{\bar}
\newcommand{\trace}{\operatorname{Tr}}
\newcommand{\SUM}{\displaystyle \sum}
\newcommand{\rank}[1]{\operatorname{rank}(#1)}

\newcommand{\beq}{\begin{equation}}
\newcommand{\eeq}{\end{equation}}
\newcommand{\bear}{\begin{align}}
\newcommand{\eear}{\end{align}}
\newcommand{\bea}{\begin{eqnarray}}
\newcommand{\eea}{\end{eqnarray}}

\long\def\symbolfootnote[#1]#2{\begingroup%
\def\thefootnote{\fnsymbol{footnote}}\footnote[#1]{#2}\endgroup}

\newtheorem*{extension}{{\bf{Definition:}} Extension}
\newtheorem*{RICDef}{{\bf{Definition:}} Restricted Isometry Constant (RIC)}
\newtheorem*{ROCDef}{{\bf{Definition:}} Restricted Orthogonality Constant (ROC)}

\newtheorem*{SSPDef}{{\bf{Definition:}} Spherical Section Property (SSP)}

\newtheorem*{RIPimply}{{\bf{RIP implications for $k$-sparse recovery}}}

\IEEEoverridecommandlockouts

\title{A Simplified Approach to Recovery Conditions for Low Rank Matrices}

\author{\IEEEauthorblockN{Samet Oymak\IEEEauthorrefmark{2}, Karthik Mohan\IEEEauthorrefmark{3}, Maryam Fazel\IEEEauthorrefmark{3} and Babak Hassibi\IEEEauthorrefmark{2}}\\
\IEEEauthorblockA{\IEEEauthorrefmark{2}   California Institute of Technology\\
{$\{$soymak, hassibi$\}$@caltech.edu}}\\
\IEEEauthorblockA{\IEEEauthorrefmark{3} University of Washington, Seattle \\
{$\{$karna, mfazel$\}$@uw.edu}}
\thanks{This work is \IEEEauthorrefmark{2}supported in part by the National Science Foundation under CCF-0729203, CNS-0932428 and CCF-1018927 and \IEEEauthorrefmark{3}supported in part by NSF CAREER grant ECCS-0847077.}
}

\begin{document}

\maketitle

\begin{abstract}
Recovering sparse vectors and low-rank matrices from noisy linear measurements has been the focus of much recent research. Various reconstruction algorithms have been studied, including $\ell_1$ and nuclear norm minimization as well as $\ell_p$ minimization with $p<1$. These algorithms are known to succeed if certain conditions on the measurement map are satisfied. Proofs for the recovery of matrices have so far been much more involved than in the vector case.
 
In this paper, we show how several classes of recovery conditions can be extended from vectors to matrices in a simple and transparent way, leading to the best known restricted isometry and nullspace conditions for matrix recovery. Our results rely on the ability to ``vectorize" matrices through the use of a key singular value inequality.
\end{abstract}
\IEEEpeerreviewmaketitle

\section{Introduction}

Recovering sparse vectors and low-rank matrices from noisy linear measurements, with applications in compressed sensing and machine learning, has been the focus of much recent research.
The Restricted Isometry Property (RIP) was introduced by Cand\`es and Tao in \cite{Candes1,Candes2} and has played a major role in proving recoverability of sparse signals from compressed measurements. The first recovery algorithm that was analyzed using RIP was $\ell_1$ minimization in \cite{Candes1,Candes2}. Since then, many algorithms including GraDes \cite{GK-09-GraDeS}, Reweighed $\ell_1$ \cite{Dea-09-Rel1},
and CoSaMP \cite{DT-08-COS} have been analyzed using RIP.
Analogous to the vector case, RIP has also been used in the analysis of algorithms for \emph{low rank matrix recovery}, for example Nuclear Norm Minimization \cite{Recht}, SVP \cite{MJD-09-SVP}, Reweighted Trace Minimization \cite{Fazel} and AdMiRA \cite{LB-09}.
Other recovery conditions have also been proposed for recovery of both sparse vectors and low-rank matrices including the Null Space Property \cite{YZ-08,Oymak} and the Spherical Section Property
\cite{YZ-08,djfazel-10} (also known as the `almost Euclidean' property) for the nullspace. The first matrix RIP result was given in \cite{Recht} where it was shown that the RIP is sufficient for low rank recovery using nuclear norm minimization, and that it holds with high probability as long as number of measurements are sufficiently large.
This analysis was improved in \cite{Candes3} to require a minimal order of measurements.
Recently, \cite{Fazel} improved the RIP constants with a stronger analysis similar to \cite{Cai}.

In this paper, we show that if a set of conditions are sufficient for the robust recovery of sparse vectors with sparsity at most $k$, then the ``extension'' (defined later) of the same set of conditions are sufficient for the robust recovery of low rank matrices up to rank $k$. 

%In particular, we show that RIP conditions that are sufficient for vector recovery are also sufficient for matrix recovery.

While the recovery analysis in \cite{Fazel} and \cite{Candes3} (Theorem 2.4) is complicated and lengthy, our results (see ``Main Theorem'') are easily derived due to the use of a key singular value inequality (Lemma \ref{singval}). Our results also apply to generic recovery conditions, one of which is RIP. As an example, $\delta_k < 0.307$ \cite{Cai}, $\delta_{2k} < 0.472$ \cite{CWX-09-Shift} are two of the many RIP-based conditions (see also \cite{Candes2}, \cite{FL-09-lq}) that are known to be sufficient for sparse vector recovery using $\ell_1$ minimization. A simple consequence of this paper is the following: The RIP conditions $\delta_k < 0.307, \delta_{2k} < 0.472$ (and all other RIP conditions that are sufficient for sparse vector recovery) are \emph{also} sufficient for robust recovery of matrices with rank at most $k$ improving the previous best condition of $\delta_{2k} < 0.307$ in \cite{Fazel}.
Improving the RIP conditions is a direct consequence of our observation but is not the focus of this paper, although such improvements have been of independent interest (e.g., \cite{Cai,CWX-09-Shift}).

Our results also apply to another recovery condition known as the Nullspace Spherical Section Property (SSP) and it easily follows from our main theorem that the spherical section constant
$\Delta >4k$ is sufficient for the recovery of matrices up to rank $k$ as in the vector case \cite{YZ-08}. This approach not only simplifies the analysis in \cite{djfazel-10}, but also gives a better condition (as compared to $\Delta>6k$ in \cite{djfazel-10}).
Our final contribution is to give nullspace based conditions for recovery of low-rank matrices using Schatten-$p$ quasi-norm minimization, which is analogous to $\ell_p$ minimization with $0<p<1$ for vectors. These nonconvex surrogate functions have motivated algorithms such as IRLS \cite{MF-IRLSC-2010,Daubechies-10} that are empirically observed to improve on the recovery performance of $\ell_1$ and nuclear norm minimization.

\section{Basic Definitions and Notation}
\label{sec:defs}

For a vector $\x\in\R^n$, $\|\cdot\|_0$ denotes its cardinality or the number of nonzero elements.
$\sort{\x}$ denotes the vector obtained by decreasingly sorting the absolute values of the entries of $\x$, and $\x^k$ denotes the vector obtained by restricting $\x$ to its $k$ largest elements (in absolute value).
Let $\td(\cdot):\R^{n\times n}\rightarrow \R^{n}$ return the vector of diagonal entries of a matrix, and  $\td(\cdot):\R^n\rightarrow \R^{n\times n}$ return a diagonal matrix with the entries of the input vector on the diagonal.
Let $n_1\leq n_2$.
We denote rank of a matrix $X\in\R^{n_1\times n_2}$ by $\rank X$, and its $i$th largest singular value by $\sig_i(X)$. Let $\Sigma(X)=[\sig_1(X),\dots,\sig_{n_1}(X)]^T$ be vector of decreasingly sorted singular values of $X$. $X^k$ denotes the matrix obtained by taking the first $k$ terms in the singular value decomposition of $X$.
The nuclear norm of $X$ is denoted by $\|X\|_\s=\sum_{i=1}^{n_1}\sig_i(X)$, and its Frobenius norm by  $\|X\|_F=\sqrt{\sum_{i=1}^{n_1}\sig_i^2(X)}$. Let $\Sigma_X=\td(\Sigma(X))\in\R^{n_1\times n_1}$. We call $(U,V)$ a unitary pair if $U^TU=UU^T=V^TV=I$. In this paper, we'll use the following singular value decomposition of $X$: $X=U\Sigma_X V^T$ where $(U,V)$ is a unitary pair. Clearly $U\in\R^{n_1\times n_1},V\in\R^{n_2\times n_1}$. Notice that the set of matrices $UDV^T$, where $D$ is diagonal, form an $n_1$ dimensional subspace. Denote this space by $S(U,V)$.
Let $\A(\cdot): R^{n_1\times n_2}\rightarrow \R^m$ be a linear operator. $\A_{U,V}(\cdot):\R^{n_1}\rightarrow \R^m$ is called the restriction of $\A$ to unitary pair $(U,V)$ if we have $\A_{U,V}(\x)=\A(U\td(\x) V^T)$ for all $\x\in\R^{n_1}$. In particular, $\A_{U,V}(\cdot)$ can be represented by a matrix $A_{U,V}$.

Consider the problem of recovering the desired vector $\x_0\in\R^n$ with $\|\x_0\|_1 = k$ from corrupted measurements
$\y=A\x_0+\z$, with $\|\z\|_2\leq\eps$ where $\eps$ denotes the noise level, and $A\in\R^{m\times n}$
denotes the measurement matrix.
Under certain conditions, $\x_0$ can be found under certain conditions by solving the following convex problem,
\begin{eqnarray} \label{prog:lone}
\begin{array}{ll}
\mbox{minimize} & \|\x\|_1\\
\mbox{subject to} & \|A\x-{\y}\|_2\leq\eps,
\end{array}
\end{eqnarray}
where recovery is known to be robust to noise as well as imperfect sparsity.
We say $\x^*$ is \emph{as good as} $\x_0$ w.r.t $\y$ if $\|A\x^*-{\y}\|_2\leq \eps$ and $\|\x^*\|_1\leq \|\x_0\|_1$. In particular, the optimal solution of the problem \ref{prog:lone} is as good as $\x_0$.

Similarly, consider the case where the desired signal to be recovered is a low-rank matrix denoted by
%With a slight abuse of notation, let
$X_0\in\R^{n_1\times n_2}$, with $n=n_1\leq n_2$. Let $\A:\R^{n_1\times n_2}\rightarrow \R^m$ be the measurement operator.
We observe corrupted measurements ${\y}=\A(X_0)+ {\z}$ with $\|{\z}\|_2\leq \eps$.
%If $X_0$ is low rank, we solve the following convex problem
\begin{equation} \label{prog:nucnorm}
\begin{array}{ll}
\mbox{minimize} & \|X\|_\s\\
\mbox{subject to} & \|\A(X)- {\y}\|_2\leq\eps.
\end{array}
\end{equation}
Similar to the vector case, we say that $X^*$ is \emph{as good as} $X_0$ w.r.t $\y$ if $\|\A(X^*) - {\y}\|_2 \leq \eps$ and $\|X^*\|_{\s} \leq \|X_0\|_{\s}$. In particular the optimal solution to \emph{problem} \ref{prog:nucnorm} is as good as $X_0$.
We now give the definitions for certain recovery conditions on the measurement map, the Restricted Isometry Property and the Spherical Section Property.
\begin{RICDef}
The RIC of a matrix $A$ is the smallest constant $\delta_k$ for which
\beq
(1-\delta_k)\|\x\|_2^2<\|A\x\|_2^2<(1+\delta_k)\|\x\|_2^2
\eeq
holds for all vectors $\x$ with $\|\x\|_0\leq k$.

The RIC for linear operators, $\A: \R^{n_1 \times n_2} \rightarrow \R^m$ is defined similarly, with $X$ instead of $\x$, $\A$ instead of $A$ and $\rank X\leq k$ instead of $\|\x\|_0\leq k$, and $\|X\|_F$ instead of $\|\x\|_2$.
\end{RICDef}

\begin{ROCDef}
The ROC of a matrix $A$ is the smallest constant $\theta_{k,k'}$ for which
\beq
|\li<A\x,A\x'\ri>|\leq \theta_{k,k'}\|\x\|_2\|\x'\|_2
\eeq
holds for all vectors $\x,\x'$ with disjoint supports and $\|\x\|_0\leq k$ and $\|\x'\|_0\leq k'$.\\

Our definition of ROC for linear operators, $\A: \R^{n_1 \times n_2} \rightarrow \R^m$ is slightly looser than the one given in \cite{Fazel}.  For a linear operator $\A$, ROC is the smallest constant $\theta_{k,k'}$ for which
\beq
|\li<\A(X),\A(X')\ri>|\leq \theta_{k,k'}\|X\|_F\|X'\|_F
\eeq
holds for all matrices $X,X'$ such that $\rank X\leq k$, $\rank X'\leq k'$ and both column and row spaces of $X,X'$ are orthogonal, i.e., in a suitable basis we can write
$X=\begin{bmatrix} X_{11} & 0\\0 & 0 \end{bmatrix}$ and $X'=\begin{bmatrix} 0&0\\0&X'_{22} \end{bmatrix}$.
\end{ROCDef}

%\begin{RIPDef}
We say that the matrix $A$ (or operator $\A$) satisfies \textbf{RIP} if the corresponding RIC and ROC constants are sufficiently small to ensure recovery of sparse vectors (low rank matrices).
As one example if a matrix $A$ satisfies $\delta_k < 0.307$ or $\delta_{2k} < 0.472$, then we say $A$ satisfies RIP, since these conditions are sufficient
for recovery of $k$-sparse vectors using $\ell_1$ minimization.
%\end{RIPDef}

\begin{SSPDef}
The Spherical Section constant of a linear operator $\A: \R^{n_1 \times n_2} \rightarrow \R^m$ is defined as
\[
\Delta(\A) = \displaystyle \min_{Z \in \N(\A)\backslash \{0\}} \frac{\|Z\|_{\s}^2}{\|Z\|_F^2},
\]
and we say $\A$ satisfies the $\Delta$-Spherical Section Property if $\Delta(\A)\geq \Delta$.
\end{SSPDef}
The definition of SSP for a matrix $A \in \R^{n\times m}$ for analyzing recovery of sparse vectors is analogous to the above definition. Another way to describe this property is to note that a large $\Delta$ implies the nullspace is an \emph{almost Euclidean} subspace \cite{YZ-08}, whose elements
cannot have a small ratio of $\ell_1$ (nuclear norm) to $\ell_2$ (Frobenius norm), and therefore the subspace cannot include sparse (low-rank) elements.\\
In the subsequent sections, we give our main results relating vector recovery to matrix recovery. 

\section{Key Observations}
Throughout this note, many of the proofs involving matrices apply the following useful Lemma which enables us to ``vectorize'' matrices when dealing with matrix norm inequalities.
\begin{lem}
\label{singval} (\cite{HJ-90}) {\bf{(Key Lemma)}}
For any $X,Y\in\R^{n_1\times n_2}$, we have
\begin{eqnarray}
\sum_{i=1}^{n_1}|\sig_i(X)-\sig_i(Y)|\leq \| X-Y \|_{\s}.
\end{eqnarray}
\end{lem}

\noindent We now give an application of Lemma \ref{singval}.
\begin{lem}
\label{lem0}
Given $W$ with singular value decomposition $U\Sigma_W V^T$, if there is an $X_0$ for which
$\|X_0+W\|_\s\leq \|X_0\|_\s$ then there exists $X_1\in S(U,V)$ with $\Sigma_{X_1}=\Sigma_{X_0}$ such that $\|X_1+W\|_\s\leq \|X_1\|_\s$. In particular this is true for $X_1=-U\Sigma_{X_0} V^T$.
\end{lem}
%there is also an $M_1=U\Sigma_1 V^T$ where $\Sigma_1$ is diagonal and $\Sigma(M)=\Sigma(M_1)$ and %$\|M_1+W\|_\s\leq \|M_1\|_\s$.
\begin{proof}
From Lemma \ref{singval} we have
\beq
\label{Mbound}
\|X_0+W\|_\s\geq\sum_i |\sig_i(X_0)-\sig_i(W)|.
\eeq
On the other hand, for $X_1 =-U\Sigma_{X_0} V^T,W$ we have
\beq
\label{M1bound}
\|X_1+W\|_\s=\|-\Sigma_{X_0}+\Sigma_W\|_\s=\sum_i|\sig_i(X_0)-\sig_i(W)|.
\eeq
Then from (\ref{Mbound}) and (\ref{M1bound}) it follows
\beq
\|X_1+W\|_\s\leq \|X_0+W\|_\s\leq \|X_0\|_\s=\|X_1\|_\s.
\eeq
\end{proof}

Although Lemma \ref{lem0} is easy to show, it proves helpful in Theorem 1 to connect vector recovery (on restricted subpaces $S(U,V)$) to matrix recovery over all space $\R^{n_1 \times n_2}$. 
%It suggests that if there exists a ``bad'' $X_0$ for a particular perturbation $W$ (i.e. $\|X_0+W\|_\s\leq \|X_0\|_\s$) , then there is also a ``bad'' $X_1$, which has the same singular values as $X_0$ and lies on the same restricted subspace $S(U,V)$ as $W$.

To further illustrate the similarity between sparse and low rank recovery, we state the null space conditions for noiseless recovery.
\begin{lem} (\cite{Cohen-Devore-09,YZ-08}) {\bf{Null space condition for sparse recovery}}\\
\label{lem1}
Let $A\in\R^{m\times n}$ be a measurement matrix. Assume $\eps=0$ then one can perfectly recover all vectors $\x_0$ with $\|\x_0\|_0\leq k$ via program \ref{prog:lone} if and only if for any $\w\in\N(A)$ we have
\beq
\label{nullcond}
\sum_{i=1}^k\sort{w}_i<\sum_{i=k+1}^n \sort{w}_i,
\eeq
where $\sort{w}_i$ is the $i$th entry of $\sort{\w}$ defined previously in section~\ref{sec:defs}.% (as $\w$ sorted by absolute values of entries).
\end{lem}
The use of the above condition has a long history; it was stated in \cite{DonohoHuo01} (see also \cite{EladBruckstein02,FeuerNemirovsky03}) for matrices made from concatenation of two bases, and studied in \cite{Cohen-Devore-09,YZ-08} in a general setting.

\begin{lem} (\cite{Oymak}) {\bf{Null space condition for low-rank recovery}} \label{lem2}
Let $\A:\R^{n_1\times n_2}\rightarrow\R^m$ be a linear measurement operator. Assume $\eps=0$ then one can recover all matrices $X_0$ with $\rank{X_0} \leq k$ via program \ref{prog:nucnorm} if and only if for any $W\in\N(\A)$ we have
\beq
\sum_{i=1}^k \sig_i(W)<\sum_{i=k+1}^{n_1} \sig_i(W).
\eeq
\end{lem}

\section{Main Result}
In this section, we state our main result which enables us to seamlessly translate results for vector recovery to matrix recovery.
Our main theorem assumes that operator $\A$ satisfies an extension property, defined below.
\begin{extension}
Let $\PP$ be a property defined for matrices, $A \in \R^{m\times n}$.
%$\R^n\rightarrow \R^m$.
We say that a linear operator, $\A: \R^{n_1 \times n_2} \rightarrow \R^m$ satisfies the extension property, $\PP^e$
if all its restrictions $A_{U,V}$ have property $\PP$.
\end{extension}
%In particular, as we discuss later, RIP and SSP conditions for matrices imply extensions of RIP and %SSP based for vectors,
%thus enabling translation of results from vectors to matrices as an application of the main theorem.
RIP and SSP are two examples of property $\PP$ that are of interest in this paper. We show later that the main theorem can be applied to these properties.

Let $\|\cdot\|_v$ be an arbitrary norm on $\R^n$ with $\|\x\|_v=\|\bar{\x}\|_v$ for all $\x$. Let $\|\cdot\|_m$ be the corresponding unitarily invariant matrix norm on $\R^{n_1\times n_2}$ such that $\|X\|_m=\|\Sigma(X)\|_v$.
For the sake of clarity, we use the following shorthand notation in the main theorem for statements regarding recovery of vectors:
%\begin{IEEEeqnarray}{rcl}
\begin{itemize}
\item $V_1$: A matrix $A:\R^n \rightarrow \R^m$ satisfies a property $\PP$.
\item $V_2$: In program \ref{prog:lone}, for any $\x_0$, $\|\z\|_2\leq\eps$, $\y = A\x_0 + z$ and any $\x^*$ as good as $\x_0$ we have,
\begin{IEEEeqnarray}{rcl}
\|\x^*-\x_0\|_v \;&\leq&\; h(\bar{{\x}}_0,\eps).
\end{IEEEeqnarray}
for some $h$.
\item $V_3$: For any $\w \in \N(A)$, $\w$ satisfies a property $\QQ$.
\end{itemize}
We also use the following shorthand for statements regarding recovery of matrices:

\begin{itemize}
\item $M_1$: A linear operator $\A:\R^{n_1 \times n_2} \rightarrow \R^m$ satisfies the extension property $\PP^e$.
\item $M_2$: In program \ref{prog:nucnorm}, for any $X_0$, $\|{\z}\|_2\leq\eps$, ${\y} = \A(X_0) + {\z}$ and any $X^*$ as good as $X_0$ we have,
\begin{IEEEeqnarray}{rcl}
\|X^*- X_0\|_m \;&\leq&\; h(\Sigma(X_0),\eps).
\end{IEEEeqnarray}
\item $M_3$: For any $W \in \N(\A)$, $\Sigma(W)$ satisfies property $\QQ$.
\end{itemize}

%Throughout this section we consider properties $\PP$ that are well defined for both matrices, $A$ and operators $\A$. We call a property $\PP$ \emph{extendable} if whenever a linear operator, $\A$ satisfies property $\PP$, all its restrictions, $\A_{U,V}$ also satisfy property $\PP$. In this section, we state our main result which enables us to translate ``extendable" recovery properties from vectors to matrices.
%In particular, as we discuss later, RIP and SSP are extendable properties, thus results using these properties extend from vectors to matrices as an application of the main theorem.
%Let $\|\cdot\|_v$ be an arbitrary norm on $\R^n$ with $\|\x\|_v=\|\bar{\x}\|_v$ for all $\x$. Let $\|\cdot\|_m$ be the corresponding matrix norm on $\R^{n_1\times n_2}$ such that $\|X\|_m=\|\Sigma(X)\|_v$.

\begin{thm} {\bf{(Main Theorem)}}\\
\label{mainthm}
For a given $\PP$, the following implications hold true:\\
\vspace{-15pt}
\begin{align}
\label{pprt1}
(V_1 \implies V_2) \implies (M_1 \implies M_2)
\end{align}
\vspace{-25pt}
\begin{align}
\label{pprt2}
(V_1 \implies V_3) \implies (M_1 \implies M_3)
\end{align}
\end{thm}

\begin{proof}
$(V_1 \implies V_2) \implies (M_1 \implies M_2)$.\\
Assume $V_1 \implies V_2$ and that $M_1$ holds. Consider program \ref{prog:nucnorm} where measurements are ${\y}_0 = \A(X_0) + {\z_0}$ with $\|\z_0\|_2 \leq \eps$. We would like to show that for any $X^*$ that is as good as $X_0$ w.r.t ${\y}_0$, it holds that $\|X^*- X_0\|_m \;\leq\; h(\Sigma(X_0),\eps)$. Consider any such $X^*$ and let $W = X^* - X_0$. This implies that $\|X_0 + W\|_{\s} \leq \|X_0\|_{\s}$ and $\|\A(X_0 + W) - {\y}_0\|_2 \leq \eps$. Then, from Lemma \ref{lem0} for $X_1 = -U\Sigma_{X_0}V^T$ (where $W$ has SVD $U\Sigma_WV^T$) we have $\|X_1 + W\|_{\s} \leq \|X_1\|_{\s}$. Now, let $\y_1=\A(X_1)+ z_0$. Clearly
\beq
\|\A(X_1 + W) - {\y}_1\|_2=\|\A(X_0 + W) - {\y}_0\|_2 \leq \eps
\eeq
and hence $X_1+W$ is as good as $X_1$ w.r.t $\y_1$. Now consider program 1 with $A_{U,V}$ as measurement matrix, $\y_1$ as measurements, $\x_1=-\Sigma_{X_0}$ as unknown vector and $\w=\Sigma_W$ as the perturbation. Notice that $\x_1+\w$ is as good as $\x_1$ w.r.t $\y_1$. Also since $\A$ has $\PP^e$, $A_{U,V}$ has $\PP$ and thus $V_1$ holds for $A = A_{U,V}$. Using $V_1\implies V_2$ we conclude
\beq
\|W\|_m=\|\w\|_v \leq h({\bar{\x}}_1,\eps)=h(\Sigma(X_0),\eps)
\eeq
It thus holds that $M_1\implies M_2$.

Using similar arguments, we now show that $(V_1 \implies V_3) \implies (M_1 \implies M_3)$.\\
Assume $V_1 \implies V_3$ and that $M_1$ holds. Consider any $W \in \N(\A)$ with SVD of $W = U\Sigma_WV^T$. Then, $\A(W) = A_{U,V}\Sigma(W) = 0$. Also $A_{U,V}$ satisfies $\PP$ and hence $V_1$ holds for $A = A_{U,V}$. Using $V_1 \implies V_3$, we find $\Sigma(W)$ satisfies $\QQ$. Hence $M_1\implies M_3$.
\end{proof}

As can be seen from the Main Theorem, throughout the paper, we are dealing with a \emph{strong} notion of recovery. By strong we mean $\PP$ guarantees recovery for all $\x_0$ ($X_0$) with sparsity (rank) at most $k$ instead of for just a particular $\x_0$ ($X_0$). For example, results for the matrix completion problem in the literature don't have a \emph{strong} recovery guarantee. On the other hand it is known that (good) RIP or SSP conditions guarantee recoverability for all vectors and yield \emph{strong} recovery results.

In the next section, we show that the main theorem can be applied to RIP and SSP based recovery and thus the corresponding recovery results for matrix recovery easily follow.

\subsection{Application of Main Theorem to RIP based recovery}
We say that $f(\delta_{i_1},\dots,\delta_{i_m},\theta_{j_1,j'_1},\dots,\theta_{j_n,j'_n})\leq c$ is an RIP inequality where $c\geq 0$ is a constant and $f(\cdot)$ is an increasing function of its parameters (RIC and ROC) and $f(0,\dots,0)=0$. Let $\F$ be a set of RIP inequalities namely $f_1,\dots,f_N$ where $k$'th inequality is of the form:
\beq
f_k(\delta_{i_{k,1}},\dots,\delta_{i_{k,m_k}},\theta_{j_{k,1},j'_{k,1}},\dots,\theta_{j_{k,n_k},j'_{k,n_k}})\leq c_k.
\eeq

\begin{lem}
\label{sameRIP}
If $\A:\R^{n_1\times n_2}\rightarrow\R^m$ satisfies a set of (matrix) RIP and ROC inequalities $\F$, then for all unitary pairs $(U,V)$, $A_{U,V}$ will satisfy the same inequalities.
\end{lem}
%XXX unitary or partial unitary? (U, V slightly different from initial SVD definition?)

\begin{proof}
Let $\delta'_k,\theta'_{k,k'}$ denote RIC and ROC of of $A_{U,V}$ and $\delta_k,\theta_{k,k'}$ denote RIC and ROC of $\A(\cdot)$. Then we claim:
$\delta'_k\leq \delta_k$ and $\theta'_{k,k'}\leq \theta_{k,k'}$.
For any $\x$ of $\|\x\|_0\leq k$, let $X=U\td(\x)V^T$. Using $\|\x\|_2=\|X\|_F$ and $A_{U,V}\x=\A(X)$ we have:
\beq
\nonumber(1-\delta_k)\|\x\|_2^2<\|\A(X)\|_F^2=\|A_{U,V}\x\|^2_2<(1+\delta_k)\|\x\|_2^2
\eeq
Hence $\delta'_k\leq \delta_k$. Similarly let $\x,\x'$ have disjoint supports with sparsity at most $k,k'$ respectively. Then obviously $X=U\td(\x)V^T$ and $X'=U\td(\x')V^T$ satisfies the condition in ROC definition. Hence:
\begin{align}
\nonumber|\li<A_{U,V}\x,A_{U,V}\x'\ri>|&=|\li<\A(X),\A(X')\ri>|\\
\nonumber&\leq \theta_{k,k'}\|X\|_F\|X'\|_F=\theta_{k,k'}\|\x\|_2\|\x'\|_2
\end{align}
Hence $\theta'_{k,k'}\leq \theta_{k,k'}$.
Thus, $A_{U,V}$ satisfies the set of inequalities $\F$ as $f_i(\cdot)$'s are increasing function of $\delta_k$'s and $\theta_{k,k'}$'s.
\end{proof}
\vspace{2pt}
We can thus combine Lemma \ref{sameRIP} and the main theorem to smoothly translate any implication of RIP for vector recovery to corresponding implication for matrix recovery. In particular, some typical RIP implications are as follows.
\begin{RIPimply} (\cite{Candes2})
\label{imp1}
Suppose $A:\R^n\rightarrow\R^m$ satisfies, a set of RIP inequalities $\F$. Then for all $\x_0$, $\|\z\|_2\leq\eps$ and $\x^*$ as good as $\x_0$ we have the following $\ell_2$ and $\ell_1$ robustness results,
%\begin{itemize}
%\item {\bf{$\ell_2$-Robustness}}
%\beq
%\label{rip1}
\begin{IEEEeqnarray}{lcl}
%\begin{eqnarray}
%\begin{array}{rcl}
\|\x_0-\x^*\|_2 \;&\leq&\; \frac{C_1}{\sqrt{k}}\|\x_0-\x_0^k\|_1+C_2\eps \\
%\eeq
%\item {\bf{$\ell_1$-Robustness}} Let $\eps=0$.
%\beq
%\label{rip2}
\|\x_0-\x^*\|_1 \;&\leq&\; C_3\|\x_0-\x_0^k\|_1
%\eeq
%\end{itemize}
%\end{array}
\label{eq:l2l1}
%\end{eqnarray}
\end{IEEEeqnarray}
for some constants $C_1,C_2,C_3>0$.
\end{RIPimply}

Now, using Lemma \ref{sameRIP} and Theorem \ref{mainthm}, we have the following implications for matrix recovery.
\begin{lem}
Suppose $\A:\R^{n_1\times n_2}\rightarrow\R^m$ satisfies the same inequalities $\F$ as in (\ref{eq:l2l1}). Then for all $X_0$, $\|{\z}\|_2\leq\eps$ and $X^*$ as good as $X_0$ we have the following Frobenius norm and nuclear norm robustness results,
\begin{IEEEeqnarray}{lcl}
%\begin{eqnarray}
%\begin{array}{rcl}
%\beq
\|X_0-X^*\|_F \;&\leq&\; \frac{C_1}{\sqrt{k}}\|X_0-X_0^k\|_\s+C_2\eps \\
%\eeq
%\item {\bf{Nuclear Norm Robustness}} Let $\eps=0$.
%\beq
\|X_0-X^*\|_\s \;&\leq&\; C_3\|X_0-X_0^k\|_\s
%\eeq
%\end{itemize}
%\end{array}
%\end{eqnarray}
\end{IEEEeqnarray}
\end{lem}

\subsection{Application of Main theorem to SSP based recovery}
\begin{lem}
Let $\Delta>0$. If $\A:\R^{n_1\times n_2}\rightarrow\R^m$ satisfies $\Delta$-SSP, then for all unitary pairs $(U,V)$, $A_{U,V}$ satisfies $\Delta$-SSP.
\label{sameSSP}
\end{lem}
\begin{proof}
%\emph{
Consider any $\w \in \N(\A_{U,V})$. Then, $\A(U\td(\w)V^T) = 0$ and therefore $W = U\td(\w)V^T \in \N(\A)$. Since $\A$ satisfies $\Delta$-SSP,
we have $\frac{\|W\|_*}{\|W\|_F} = \frac{\|\w\|_1}{\|\w\|_2} \geq \sqrt{\Delta}$. Thus $A_{U,V}$ satisfies $\Delta$-SSP.
%}
\end{proof}
\vspace{3pt}
Now we give the following SSP based result for matrices as an application of main theorem.
\begin{thm}
Consider program \ref{prog:nucnorm} with $\z=0$, $\y = \A(X_0)$. Let $X^*$ be as good as $X_0$. Then if $\A$ satisfies $\Delta$-SSP with $\Delta  > 4k$, it holds that
\begin{IEEEeqnarray}{rcl}
\|X^* - X_0\|_{\s} \;&\leq&\; C\|X_0 - X^k_{0}\|_\s
\end{IEEEeqnarray}
where $C = \frac{2}{1 - 2{\sqrt{k/\Delta}}}$.
\end{thm}
Note that the use of main theorem and \emph{Key Lemma} simplifies the recovery analysis in \cite{djfazel-10} and also improves the sufficient condition of
$k < \frac{\Delta}{6}$ in \cite{djfazel-10} to $k < \frac{\Delta}{4}$. This improved sufficient condition matches the sufficient condition given in \cite{YZ-08} for the sparse vector recovery  problem.
%\vspace{-3pt}
\section{Simplified Robustness Conditions}
We show that various robustness conditions are equivalent to simple conditions on the measurement operator. The case of noiseless and perfectly sparse signals, is already given in Lemmas \ref{lem1} and \ref{lem2}. Such simple conditions might be useful for analysis of nuclear norm minimization in later works.
We state the conditions for matrices only;
%as condition for vectors most probably exist in prior literature.
however, vector and matrix conditions will be identical (similar to Lemmas \ref{lem1}, \ref{lem2}) as one can expect from Theorem \ref{mainthm}. The proofs follow from simple algebraic manipulations with the help of Lemma \ref{singval}.

\begin{lem} {\bf{(Nuclear Norm Robustness for Matrices)}}\\
\label{lem4}
Assume $\eps=0$ (no noise). Let $C>1$ be constant. Then for any $X_0$ and any $X^*$ as good as $X_0$ we are guaranteed to have:
\beq
\|X_0-X^*\|_\s<2C\|X_0-X_0^k\|_\s
\eeq
if and only if for all $W\in\N(\A)$ we have:
\beq
\|W^k\|_\s< \frac{C-1}{C+1}\|W-W^k\|_\s
\eeq
%\begin{proof}
%It follows from a simple combination of arguments for Lemmas \ref{lem2}, \ref{lem3}.
%\end{proof}
\end{lem}
\begin{lem} {\bf{(Frobenius Norm Robustness for Matrices)}}\\
\label{lem6}
Let $\eps=0$. Then for any $X_0$ and $X^*$ as good as $X_0$,
\beq
\|X_0-X^*\|_F< \frac{C}{\sqrt{k}} \|X_0-X_0^k\|_\s,
\eeq
if and only if for all $W\in\N(\A)$, %with $\sum_{i=1}^k\sig_i(W)\geq \sum_{i=k+1}^n\sig_i(W)$
\beq
\|W-W^k\|_\s-\|W^k\|_\s > \frac{2\sqrt{k}}{C} \|W\|_F
\eeq
\end{lem}

\begin{lem} {\bf{(Matrix Noise Robustness)}}
\label{lem8}
For any $X_0$ with $r(X_0)\leq k$, any $\|\z\|_2\leq\eps$ and any $X^*$ as good as $X_0$,
\beq
\|X_0-X^*\|_F< C\eps,
\eeq
if and only if for any $W$ with $\|W^k\|_\s\geq\|W-W^k\|_\s$,
\beq
\|W\|_F<\frac{C}{2}\|\A(W)\|_2
\eeq
\end{lem}

\section{Null space based recovery result for Schatten-$p$ quasi-norm minimization}
In the previous sections, we stated the main theorem and considered its applications on RIP and SSP based conditions to show that results for recovery of sparse vectors can be analogously extended to recovery of low-rank matrices without making the recovery conditions stronger. In this section, we consider extending results from vectors to matrices using an algorithm different from $\ell_1$ minimization or nuclear norm minimization.

The $\ell_p$ quasi-norm (with $0<p<1$) is given by $\|x\|_p^p = \sum_{i=1}^n |x|_i^p$. Note that for $p = 0$, this is nothing but the cardinality function. Thus it is natural to consider the minimization of the $\ell_p$ quasi-norm (as a surrogate for minimizing the cardinality function). Indeed, $\ell_p$ minimization has been a starting point for algorithms including Iterative Reweighted Least Squares \cite{Daubechies-10} and Iterative Reweighted $\ell_1$ minimization \cite{FL-09-lq,CWB-08-Rel1}. Note that although $\ell_1$ minimization is convex, $\ell_p$ minimization with $0 < p < 1$ is \emph{non-convex}. However empirically, $\ell_p$ minimization based algorithms with $0 < p < 1$ have a better recovery performance as compared to $\ell_1$ minimization (see e.g. \cite{chartrand-2008-iteratively},\cite{FL-09-lq}). The recovery analysis of these algorithms has mostly been based on RIP. However Null space based recovery conditions analogous to those for $\ell_1$ minimization have been given for $\ell_p$ minimization (see e.g. \cite{GribonvalNielsen03, MengWang-10}).

Let $\trace |A|^p = \trace(A^TA)^{\frac{p}{2}} = \sum_{i=1}^n \sigma_i^p(A)$ denote the \emph{Schatten-$p$ quasi norm} with $0 < p < 1$. Analogous to the vector case, one can consider the minimization of the Schatten-$p$ quasi-norm for the recovery of low-rank matrices,
\begin{eqnarray}
\begin{array}{rc}
\mbox{minimize} & \trace |X|^p \\
\mbox{subject to} & \A(X) = y
\end{array}
\label{eq:Qpnorm}
\end{eqnarray}
where $y = \A(X_0)$ with $X_0$ being the low-rank solution we wish to recover.
IRLS-$p$ has been proposed as an algorithm to find a local minimum to (\ref{eq:Qpnorm}) in \cite{MF-IRLSC-2010}. However no null-space based recovery condition has been given for the recovery analysis of Schatten-$p$ quasi norm minimization.
%We give below a Null Space based recovery condition for recovery of low-rank matrices using Schatten-$p$ % quasi-norm minimization.
We give such a condition below, after mentioning a few useful inequalities.
\vspace{-3pt}
\begin{lem}[\cite{Uchiyama-05}]
For any two matrices $A,B \in \R^{m\times n}$ it holds that $\sum_{i=1}^k (\sigma_i^p(A) - \sigma_i^p(B)) \leq \sum_{i=1}^k \sigma_i^p(A-B)$ for all $k = 1,2,\ldots,n$.
\label{lemma:pmaj}
\end{lem}
\vspace{-1pt}
Note that the $p$ quasi-norm of a vector satisfies the triangle inequality ($x,y \in \R^n$, $\sum_{i=1}^n |x_i + y_i|^p \leq \sum_{i=1}^n |x_i|^p + \sum_{i=1}^n |y_i|^p$). Lemma \ref{lemma:pmaj} generalizes this result to matrices.
\vspace{-3pt}
\begin{lem}[\cite{HJ-91}]
For any two matrices $A,B$ it holds that
\begin{eqnarray}
\sigma_{t+s-1}(A + B)  \leq \sigma_{t}(A) + \sigma_{s}(B)
\end{eqnarray}
where $t+s-1 \leq n, t,s \geq 0$.
\end{lem}
\vspace{-3pt}
The following lemma easily follows as a consequence.
\vspace{-3pt}
\begin{lem}
For any two matrices, $A,B$ with $B$ of rank $k$ and any $p > 0$,
\[
\sum_{i=k+1}^{n-k} \sigma_i^p(A - B) \geq \sum_{i=2k+1}^n \sigma_i^p(A).
\]
\label{lemma:lowertri}
\end{lem}
%$\tr{(X_0)}$
\begin{thm}
\label{thmp}
Let \emph{rank}${(X_0)} = k$ and let $\bar{X}$ denote the global minimizer of (\ref{eq:Qpnorm}). A sufficient condition for $\bar{X} = X_0$ is that $\sum_{i=1}^{2k}\sigma_i^p(W) \leq \sum_{i=2k+1}^n\sigma_i^p(W)$ for all $W \in \N(\A)$. A necessary condition for $\bar{X} = X_0$ is that $\sum_{i=1}^k\sigma_i^p(W) \leq \sum_{i=k+1}^n\sigma_i^p(W)$ for all $W \in \N(\A)$.
\label{theorem:NSPpnorm}
\end{thm}
\begin{proof}
($\Rightarrow$) For any $W \in \N(\A)\backslash \{0\}$,
\begin{IEEEeqnarray}{rcl}
%\begin{array}{rcl}
\trace |X_0 + W|^p &=&\; \SUM_{i=1}^k \sigma_i^p(X_0 + W) + \SUM_{i=k+1}^n \sigma_i^p(X_0 + W) \nonumber\\
								&\geq&\; \SUM_{i=1}^k (\sigma_i^p(X_0) - \sigma_i^p(W)) + \SUM_{i=2k+1}^n  \sigma_i^p(W) \nonumber \\
								&\geq&\; \trace |X_0|^p
%\end{array}
\end{IEEEeqnarray}							
where the first inequality follows from  Lemma \ref{lemma:pmaj} and Lemma \ref{lemma:lowertri}.
The necessary condition is easy to show analgous to the results for nuclear norm minimization.
\end{proof}

Note that there is a gap between the necessary and sufficient conditions. We observe through numerical experiments that a better inequality such as $\SUM_{i=1}^n \sigma_i^p(A - B) \geq \SUM_{i=1}^n |\sigma_i^p(A) - \sigma_i^p(B)|$ seems to hold for any two matrices $A,B$. Note that this is in particular true for $p = 1$ (Lemma \ref{singval}) and $p = 0$. If this inequality is proven true for all $0 < p < 1$, then we could bridge the gap between necessity and sufficiency in Theorem \ref{theorem:NSPpnorm}. Thus we have that singular value inequalities including those in Lemma \ref{singval} and Lemma \ref{lemma:lowertri}, \ref{lemma:pmaj} play a fundamental role in extending recovery results from vectors to matrices. Although, our condition is not tight, we can still use the Theorems \ref{mainthm} and \ref{thmp} to conclude the following:
\vspace{-3pt}
\begin{lem}
\label{plemma}
Assume property $\SSS$ on matrices $\R^n\rightarrow\R^m$ implies perfect recovery of all vectors with sparsity at most $2k$ via $\ell_p$ quasi-norm minimization where $0< p < 1$. Then $\SSS^{e}$ implies perfect recovery of all matrices with rank at most $k$ via Schatten-$p$ quasi-norm minimization.
\end{lem}
%\begin{proof}
\vspace{-3pt}
{{\emph{Proof Idea:}}} Since $\SSS$ implies perfect recovery of all vectors with sparsity at most $2k$, it necessarily implies a \emph{null space property} (call it $\QQ$) similar to the one given in Lemma \ref{lem1}, (\ref{nullcond}) but with $p$ in the exponent (see e.g. \cite{MengWang-10}) and $k$ replaced by $2k$. Now the main theorem, (\ref{pprt2}), combined with Theorem \ref{thmp} implies that $\SSS^e$ is sufficient for perfect revovery of rank $k$ matrices.
\vspace{3pt}
%\end{proof}
\\In particular, using Lemma \ref{plemma}, we can conclude, any set of RIP conditions that are sufficient for recovery of vectors of sparsity up to $2k$ via $\ell_p$ minimization, are also sufficient for recovery of matrices of rank up to $k$ via Schatten-$p$ minimization. As an immediate consequence it follows that the results in \cite{chartrand-2008-restricted,FL-09-lq} can be easily extended.

\vspace{-5pt}
\section{Conclusions}

We presented a general result stating that the extension of any sufficient condition for the recovery of sparse vectors using $\ell_1$ minimization is also sufficient for the recovery of low-rank matrices using nuclear norm minimization. Consequently, we have that the best known RIP-based recovery conditions of $\delta_k < 0.307$ (and $\delta_{2k} < 0.472$) for sparse vector recovery are  also sufficient for low-rank matrix recovery. Note that our result shows there is no ``gap" between the recovery conditions for vectors and matrices, and we do not lose a factor of 2 in the rank of matrices that can be recovered, as might be suggested by existing analysis (e.g. \cite{Recht}).

We showed that a Null-space based sufficient condition (Spherical Section Property) given in \cite{YZ-08} easily extends to the matrix case, tightening the existing conditions for low-rank matrix recovery \cite{djfazel-10}. Finally, we gave null-space based conditions for recovery using Schatten-$p$ quasi-norm minimization and showed that RIP based conditions for $\ell_p$ minimization extend to the matrix case. We note that all of these results rely on the ability to ``vectorize'' matrices through the use of key singular value inequalities including Lemma \ref{singval}, \ref{lemma:pmaj}.

%\bibliographystyle{plain}
%\bibliography{ISIT11}

\bibliographystyle{plain}
\bibliography{ISIT11}
\end{document}